\documentclass{l4dc2025}

\title[Topological State Space Inference for Dynamical Systems]{Topological State Space Inference for Dynamical Systems}
\usepackage{times}
\usepackage{wrapfig}
\usepackage{eulervm}
\usepackage{tgpagella}
\usepackage[T1]{fontenc}
\usepackage{amsmath}
\usepackage{amssymb}
\usepackage{cite}
\usepackage{color}
\usepackage{graphicx}
\usepackage{color}
\usepackage{bm}
\usepackage{mathtools}

\DeclareMathAlphabet{\mathpzc}{OT1}{pzc}{m}{it}

\usepackage{algorithm}
\usepackage{algpseudocode}
\makeatletter \def\BState{\State\hskip-\ALG@thistlm} \makeatother

\usepackage{etoolbox} \makeatletter
\patchcmd{\@maketitle}{\begin{center}}{\begin{flushleft}}{}{}
\patchcmd{\@maketitle}{\begin{tabular}[t]{c}}{\begin{tabular}[t]{@{}l}}{}{}
\patchcmd{\@maketitle}{\end{center}}{\end{flushleft}}{}{}

\newcommand\be{\begin{equation}}
\newcommand\ee{\end{equation}}

\newtheorem{thm}{Theorem}[section]

\newtheorem{lem}[thm]{Lemma}

\newtheorem{dfn}[thm]{Definition}

\newtheorem{obs}[thm]{Observation}

\newcommand{\Real}{\mathbb{R}}

\newcommand{\N}{\mathbb{N}}

\def\Real{\mathbb{R}}

\newcommand{\R}{\ensuremath{\mathbb{R}}}

\newcommand{\Mor}{\mathtt{Mor}}
\newcommand{\VR}{\bm{\Sigma}}
\newcommand{\Sph}{\bm{S}}

\author{%
 \Name{Mishal {Assif P K}} \Email{mishal2@illinois.edu}\\
 \addr Department of Electrical and Computer Engineering, Grainger College of Engineering, University of Illinois at Urbana-Champaign, Urbana, IL
 \AND
 \Name{Yuliy Baryshnikov} \Email{ymb@illinois.edu}\\
 \addr Departments of Mathematics and Electrical and Computer Engineering, Grainger College of Engineering, University of Illinois at Urbana-Champaign, Urbana, IL
}

\begin{document}

\maketitle

\begin{abstract}%
We present a computational pipe aiming at recovery of the topology of the underlying phase space from observation of an output function along a sample of trajectories of a dynamical system. \end{abstract}


\section{Introduction}

The problem of state space realization, that is, the reconstruction of the underlying space of a dynamical system and an observation function, remains foundational in systems theory. At its core, state space realization aims to reconstruct or identify a state space model that accurately describes the internal dynamics of a system based solely on observed data or system outputs. This problem originated with the work of Richard Bellman and Rudolf Kalman, whose pioneering research in the 1960s laid the groundwork for modern systems theory \citep{schutter_minimal_2000}.

The state space realization problem for linear models is well studied. Classical elegant solutions are available in this case relying, in essence on the theory of modules over the polynomial rings. Linearity enables a clear and tractable description of system dynamics, where system dynamics can be modeled as linear transformations, providing well-established methods for state space realization.

Extensions to nonlinear systems introduce, however, substantial complexity. Nonlinear systems often exhibit intricate behaviors, including bifurcations and chaos. Consequently, realizing or reconstructing a state space for such systems is far less straightforward. In this note, we present an attempt to address the nonlinear version of the state space realization problem using a mathematical theory as a guiding light, much as the LTI realization theory relied on the structure theory of modules over polynomial rings in one variable.

\subsection{Setup}
Consider a dynamical system $(M, \phi)$ where $M$ is a topological space and $\phi:\R \times M\to M$ is a one-parametric transformation group, that is, $\phi$ is a continuous function
\[
\R \times M \ni (t, m) \rightarrow \phi(t, m) := \phi_t(m) \in M \text{ such that } \phi_0(m) = m, \ \phi_t(\phi_s(m)) = \phi_{t+s}(m)
\]
In most of our examples, $M$ will be a smooth compact manifold and the transformation group will be generated by flows of a vector field $v : M \rightarrow TM$. We will denote by $\pi_x : \R \rightarrow M$ the trajectory of $v$ such that
\begin{equation}\label{eq:dynamics}
\pi_x(0) = x, \ \frac{d \pi_x(t)}{dt} = v(\pi_x(t)).
\end{equation}
Note that in such cases, the transformation group $\phi$ is at least as smooth as the vector field $v$.

The question we investigate here is {\em whether the topology of the space $M$ can be recovered from a finite number of trajectories $\{ \pi_{x_k} \}_{k=1,..,n}$} of such a dynamical system on $M$. We will see that this is indeed possible and provide a principled computational pipeline and a methodology to extract topological information of $M$ from such trajectories.

\subsection{Data Analytic setup}
We are interested in building an algorithmic setup towards inference of the topology from observation data. To make the general problem more amenable to such data analysis, we add the following constraints to the problem setting:
\begin{enumerate}
\item The data are collected at discrete time intervals, whether it be actual readings from sensors or the result of numerical algorithms or simulations. In addition, such readings are collected only up to a finite time limit. We cannot hope to get entire trajectories $\pi_x$ starting from a point $x$. Instead we will get discrete trajectories obtained by performing a finite sampling along the trajectory at regular time intervals.
The sampling interval $\delta$ will be assumed to be constant across all trajectories and will be dropped from the notation hereafter.
\item One observes only a finite sample of such trajectories, initialized according to some distribution, which may or not be under our control.
\item One observes not the points of the state space $M$ (which is unknown), but rather a smooth output function
\[
g: M \rightarrow V \cong \R^p
\]
such that only $g^n(x) := (g(\pi_x(-l\delta)), g(\pi_x(-(l-1)\delta)), ..., g(\pi_x(k\delta))) \in V^n, n=k+l+1$ is available to us. The dimension of $V$ can be much smaller than that of $M$, making this problem more challenging. 
\end{enumerate}

Hence the overall problem we investigate here is: \textit{What topological information of $M$ can be extracted from a finite set of observations of uniformly sampled trajectories $\{ g^n(\pi_{x_\alpha}) \}_{\alpha=1,..,N}$ of a dynamical system on $M$?}

\subsection{Takens' Embedding}
There is a well-known answer to the question of the state realization due to the Takens embedding theorem \citep{Tak06}.

\begin{thm}
Let $M$ be a compact smooth manifold, and $n \geq 2\dim(M)+1$ an integer. Then, for a {\em generic} pair $(\phi, g)$, where $\phi:M\to M$ is a $C^2$-diffeomorphism, and $g$ a function in $C^2(M, \R)$, the map
\[
	M \ni x \mapsto ( g(x), g(\phi(x)), ..., g(\phi^{n-1}(x)) \in \R^k
\]
is an {\em embedding} of $M$.
\end{thm}

Genericity here is understood in the standard way: the pairs $(\phi,g)$ which satisfy this condition form an open, everywhere dense subset in the Banach manifold of such pairs.

Takens embedding theorem guarantees that for generic dynamical systems and observations, given the infinite collection of observations of uniformly sampled trajectories starting at any point on the manifold $M$ with length more than twice the dimension of $M$, we can exactly recover $M$ up to a diffeomorphism. 

While Takens theorem is encouraging and suggests that our problem is not completely intractable, it does not address our problem, being a purely existence result, not providing a viable strategy for reconstructing $M$ from the trajectories, even in the ideal case where trajectories starting from all possible initial points are available.

In our setting, the nature of the sample one observes, - a finite sample, -  is intrinsically different from the smooth manifold $M$. Nevertheless, {\em topological invariants}, namely, the homology groups of the state space, can be recovered from finite observation of the trajectories, as we argue below. 

\subsection{Related Results}
The problem of the state space realization of nonlinear dynamical systems from observations has been addressed by many authors. Thus, \citep{Jac80} develops an analogue of the Taken's embedding theorem for nonlinear control systems, while \citep{schaft86} looks at the action of the dynamics on the functions in a similar manner, leading to, essentially, a version of the Koopman formalism \citep{budisic12}. 

The recent work \citep{otto2023learning} considers data driven state space recovery using Machine Learning methods. Among other works exploring such model-free approaches one should mention \citep{schmidt2009distilling, brunton2016discovering,champion2019data}.

\section{Mathematical Background}
The motivation for our theory comes from the beautiful work of Cohen, Jones, and Segal \citep{Coh95}, which we then couple with the well-established apparatus of the Topological Data Analysis (TDA) \citep{harer2010}.

\subsection{Basic Definitions}
Here is their setup. Assume that $M$ is a smooth closed Riemannian manifold and $f : M \rightarrow \R$ a smooth Morse function on $M$. The key idea of \citep{Coh95} is to describe a combinatorial
construction that reconstructs the topology of $M$ from the space of flow lines of the {\em gradient vector field} $\nabla f$. 

To this end, they introduced the enriched category\footnote{The language of category theory (which can be looked up in, say, \citep{Riehl}) can be avoided, but for the space reasons we just follow the original presentation. Enriched category here just means that the morphisms between any objects carries an extra structure (of the topological space).}$\mathcal{C}_f$. 
The objects of this category are critical points of $f$, and the space of morphisms $\Mor(a_-,a_+)$ between two critical points $a_-, a_+$ is the closure of the space of the flow lines $\gamma:\Real\to M$ of the gradient vector field $v$ such that $\gamma(t)\to a_\pm$ as $t\to \pm \infty$.

Concatenating smooth flow lines (which have matching target and source, respectively) allows one to define a continuous composition $\Mor(a,b)\times \Mor(b,c)\to \Mor(a,c)$.

The key result of \citep{Coh95} is that the {\em classifying space of the category} $\mathcal{C}_f$, denoted by $\mathcal{BC}_f$, is homotopy equivalent to $M$. 
Recall that the classifying space $\mathcal{BC}_f$ of a category is a simplicial space whose spaces of $k$-simplices are parametrized by $k$-tuples of composable morphisms in $\mathcal{C}_f$, modulo a natural identification procedure. Altogether, these identification operations yield a combinatorial procedure of gluing together the pieces (topological spaces $\Mor(\cdot, \cdot)$), resulting in a topological space, the classifying space of the category $\mathcal{C}_f$
The relation of this space to the original manifold $M$ is given by the theorem:

\begin{thm}[\citep{Coh95}]
Let $f:M \rightarrow \R$ be a Morse function on the closed Riemannian manifold $M$. Let $\mathcal{C}_f$ be the topological category whose objects are the critical points of $f$ and whose space of morphisms between two critical points is the space of piecewise flow lines of the gradient vector field of $f$ connecting the two critical points and composition of morphisms given by composition of piecewise flow lines, then
\begin{enumerate}
\item the classifying space $\mathcal{BC}_f$ is homotopy equivalent to $M$,
\(
	\mathcal{BC}_f \simeq M.
\)
\item the classifying space $\mathcal{BC}_f$ is homeomorphic to $M$ if $f$ is a (generic) Morse function whose gradient flow satisfies the Morse-Smale transversality conditions,
\(
	\mathcal{BC}_f \cong M.
\)
\end{enumerate}
\end{thm}

\subsection{Topological Data Analysis}
We give a brief introduction to the relevant tools (see \citep{harer2010} for a general reference). 

Let $X$ be a compact metric space (say, a manifold). Given a finite sample of points $\mathcal{A}=\{x_\alpha\}, \alpha\in A$ drawn from $X$ (and the induced distances between the points), persistent homology is a tool to recover the topological features (e.g., the homology groups $H_k(X)$ in dimensions $k \in \N$) of the space $X$ from the discrete data $\mathcal{A}$ alone. \footnote{The homology group in dimension $k$ is a vector space whose dimensions reflects the number of $k$-dimensional holes in $X$. Although there are other important topological features of a space, homology groups are particularly easy to compute algorithmically and so a large part of TDA focuses on these groups.}

A finite collection of open subsets $\{U_\alpha\}_{\alpha\in A}$ covering $X$ is a {\em nice} cover if any intersection of these subsets is either empty or contractible. In this situation, the classical {\em Nerve Theorem} implies that a combinatorial object, the simplicial complex formed by the subsets $S\subset A$ for which the intersection of the open sets $\cap_{\alpha\in A} U_\alpha$ is non-empty, is homotopy-equivalent to $X$, thus opening an effective procedure of computing homologies of $X$. 

A collection of open metric balls $U_\alpha=B^\circ_r(x_\alpha), \alpha\in A$, could serve as such an open collection, if one could overcome two difficulties. 

One is that detecting the non-empty intersection of those balls might be difficult. An effective palliative is to work rather than with the nerve with the {\em Vietoris-Rips} simplicial complex: here one declares a collection $S\subset A$ to form a simplex, if all {\em pairwise} distances between points $x_\alpha, \alpha\in S$ are at most $r$. This works, at least for dense enough samples,  for appropriate $r$ \citep{Lat}.

The other difficulty is to find an appropriate scale $r$: choose it too small, and there are no nontrivial simplices; too large, and the resulting space is one high-dimensional simplex.
 
To overcome this problem, one considers a {\em filtered simplicial complex}, an increasing collection of simplicial complexes $\Sigma_r, r\in \Real, \Sigma_r\subset \Sigma_{r'} \mbox{ for } r<r'$. In our context, $\Sigma_r$ is the Vietoris-Rips complex corresponding to the cutoff distance $r$.

In this setting, one can keep track of the modules of cycles and boundaries in $\Sigma_r$ as they evolve with $r$. A {\em persistent} cycle, i.e., one that emerged for small $r$ but is not patched until much larger $r$ is deemed to represent the true homology of the space $X$. The theory developed over the past two decades allows one to set up efficient computational procedures for finding them \citep{harer2010}.

\section{Inferring Topology from Observations}
We will be deploying the tools of Persistent Homology for the metrics on the spaces of samples of trajectories we define relying on the intuitive picture generalizing the Cohen-Jones-Segal (CJS) construction. In a contrast with their approach, we allow arbitrary smooth dynamics, not requiring it being gradient.

\subsection{Approximations}
The key difference of CJS setup from Takens embedding is its reliance on the infinitely long trajectories. This forces corresponding adjustments to any computational implementation.

Consider the space $\Pi_T=\{\pi_t(x)(t),x\in M, t\in [-T,T]\}$ of the fragments of trajectories of the length $2T$. They are  parameterized by their centers $x$ (so that the space is diffeomorphic to $M$). Defining the distance between the trajectories in the standard way, using $C^0$ or similar norm would essentially reproduce Takens (or Koopman) paradigm. Instead, we use the distance that accounts for proximity of long subintervals of the trajectories.

As we intend to recover the state space from observations, we need to measure rather the distances between the trajectories of {\em observed} values. These considerations lead us to the following 

\begin{dfn}[Slack distance]\label{dfn:slack}
Fix $T$, $n=2k+1$, and $\delta$ such that $T=k\delta$. For a given observation function $g:M \to V\cong \Real^d$, let
\[
y^n(x):=(g(\pi_{x}(-k\delta)),\pi_{x}(-(k-1)\delta))\ldots,g(\pi_{x}((k-1)\delta)),g(\pi_{x}(k\delta)))\in V^n.\] 
Let $y^n_1=y^n(x_1), y^n_2=y^n(x_2)$ be two elements of $V^n$ (i.e., two length $n$ samples of trajectories of the dynamical system on $M$). 

We say that $y^n_1$ and $y^n_2$ are at slack distance at most $(t, \epsilon)$ if the subsequences
\[
( y^n_1(i), y^n_1(i+1), \ldots, y^n_1(i+n-s) ) \text{ and } ( y^n_2(j), y^n_2(j+1), \ldots, y^n_2(j+n-s) ) \in V^{n-s}
\]
are at the $\sup$ distance at most $\epsilon$ for some $s\leq t$. 

We say that $x_1, x_2\in M$ are at slack distance $(t, \epsilon)$ if $y^n_1=y^n(x_1), y^n_2=y^n(x_2)$ are.
\end{dfn}
In other words, $x_1, x_2 \in M$ are at slack distance $(t, \epsilon)$ if the measurements along the trajectories $\pi_{x_1}$, $\pi_{x_2}$ are $\epsilon$-close to each other for at least $L$ steps, for some $L\geq n-t$.

\begin{lem}\label{lem:slack}
If $y^n_1, y^n_2$ are at the slack distance $(s,\epsilon)$, and $y^n_2, y^n_3$ are at the slack distance $(s',\epsilon')$, then $y^n_1, y^n_3$ are at the slack distance at most $(s+s', \epsilon+\epsilon')$.
\end{lem}
\begin{proof}
There are $s$ indices at the endpoints of $y^n_1$ and $y^n_2$ that are at distance greater than $\epsilon$ and $s'$ indices at the endpoints of $y^n_2$ and $y^n_3$ that are at distance greater than $\epsilon'$. Outside the union of these indices, which has length at most $s+s'$, the elements of pairs of trajectories $(y^n_1, y^n_2)$ and $(y^n_2, y^n_3)$ are within distance $\epsilon$ and $\epsilon'$ respectively. This means the elements of $(y_1, y_3)$ are within distance $\epsilon+\epsilon'$ outside a set of indices of length $s+s'$.
\end{proof}

As a corollary, for any convex homogeneous degree $1$ function $d_\pi$ of $s$ and $d$, $d_\pi(s,d)$ satisfied the triangle inequality.
The slack distance naturally gives rise to a bifiltration whose vertices represent trajectories of observations from a dynamical system and an edge appears between two trajectories at parameter $(k, \epsilon)$ if the trajectories are at slack distance
$(k, \epsilon)$. 



\subsection{Vietoris-Rips Complexes for Slack Distance}

Consider a sample $\mathcal{A}$ of $N$ trajectories of the dynamical system \eqref{eq:dynamics}. For each of these trajectories $\pi_{x_l}:[-T,T]\to M, l=1,\ldots, N$, we assume that an output function $g:M\to V$ is observed at equispaced intervals, resulting in a collection of observations $y^n_l$ of length $n$. Our objective of is to recover the topology of $M$ from this set of observables. 

The definition of the slack distance (and the Lemma \ref{lem:slack}) provide for a flexible construction of filtrations by Vietoris-Rips complexes with respect to slack distance.

Namely, we define the space $\VR(t, r)$ as the Vietoris-Rips complex on $\mathcal{A}$ corresponding to parameters $(t,r)$. In other words, the simplices of $\VR(t, r)$ are given by the condition 
\begin{equation*}
\VR(t,r)= \{\sigma=(\alpha_0,\ldots \alpha_d): \mbox{ slack distance between } x_{\alpha_i}, x_{\alpha_j} \mbox{ is at most } (t,\epsilon)\}.
\end{equation*}


The slack distance between two trajectories for fixed slack $t\leq n$ can be calculated in $O(n^2 \log(n))$ time using dynamic programming, and the corresponding bifiltration can be built in $O(n^2\log(n)N^2)$ time. The algorithm also requires $O(n^2)$ space. 


The pseudocode of the algorithm is shown as Algorithm \ref{alg:two} below.

\SetKwComment{Comment}{/* }{ */}

\begin{algorithm}
\small
\caption{Algorithm for building the MSD bifiltration}\label{alg:two}
\KwData{ A pair of trajectories $x^i, x^j$ both of length $n$}
\KwResult{List $\mathtt{ms\_dist}$ of length $n$ containing the Matching substring distances between $x^i$ and $x^j$}
\BlankLine
Compute distance matrix $d(k,l) = \|x^i_k - x^j_l\|$\\
$\mathtt{srtd\_indices} \leftarrow argsort(d(k,l))$\\
$\mathtt{ms\_dist}[i] = \infty \ \text{ for } i \in \{0,1,...,n-1\}$\\
\For{$i,j \in \mathtt{srtd\_indices}$}{
    $\mathtt{init\_indices}[i,j] = i,j$ \\
    $\mathtt{final\_indices}[i,j] = i,j$ \\
    $\mathtt{len} = 1$ \\
    \uIf{$i+1,j+1 \in \mathtt{init\_indices}$}{
        $\mathtt{init\_indices}[i,j] = \mathtt{init\_indices}[i+1,j+1]$ \\
        $\mathtt{final\_indices}[\mathtt{init\_indices}[i,j]] = \mathtt{final\_indices}[i,j]$ \\
        $\mathtt{len} += \mathtt{init\_indices}[i+1,j+1][0]-i$\\
    }
    \uIf{$i-1,j-1 \in \mathtt{final\_indices}$}{
        $\mathtt{final\_indices}[i,j] = \mathtt{final\_indices}[i-1,j-1]$ \\
        $\mathtt{init\_indices}[\mathtt{final\_indices}[i,j]] = \mathtt{init\_indices}[i,j]$ \\
        $\mathtt{len} += i-\mathtt{final\_indices}[i-1,j-1][0]$\\
    }
    $\mathtt{ms\_dist}[\mathtt{len}] = \min(\mathtt{ms\_dist}[\mathtt{len}], d(i,j))$
}
\end{algorithm}

\subsection{Slack Distance and Sliding Windows}
Our approach to constructing a filtration out of trajectories of observables of a dynamical system for recovering topological information is superficially related to the {\em sliding window} setup developed in \citep{Per15}. 
The key difference that makes our approach more efficient is the {\em slack} allowing measuring the distance between shifted fragments of two trajectories. 
In the sliding window approach, there is no slack allowing the trajectories to be aligned after a time shift: in some sense, the distance between points pulled back from the distance between the sampled trajectories is quasi-conformal to the original metric structure on $M$. 

In our approach, the presence of slack that makes the sampled trajectory $y^n(x)$ close to the sampled trajectory $y^n(\phi(\tau, x))$ for a bounded $\tau$ in essence adjusts the underlying metric structure to be small {\em along} the trajectories of $v$. This anisotropic rescaling of the underlying metric should be compared to the constructions of \citep{FerryO}, where the loosening of the Riemannian metric along the fibers of a map $M\to N$ leads to the Gromov-Hausdorff convergence of the resulting metric structures on $M$ to those on $N$ under some connectivity assumptions.
We conjecture that the surprising efficacy of our algorithm can be explained along those lines, - which is a separate project, to be pursued elsewhere.
\section{Computational Examples}
The rest of the note deals with the computational examples, where we attempt to reconstruct the topology of the underlying phase space, using the finite samples (of size $N$) of trajectories, finite observations of length $n$ along the trajectories, and slack distance for some slack parameter $t$.  The results are shown as persistence diagrams with respect to the allowable distance $r$ between the fragments of the trajectories.

\subsection{Gradient Dynamical System: Height Function on the 2-Sphere}
We apply our computational pipeline to a gradient dynamical system on the 2-sphere $\Sph^2 \subset \R^3$ embedded as a submanifold of $\R^3$ in the standard sense. Consider the height function \[\Sph^2 \ni (x,y,z) \mapsto h(x,y,z) = z \in \R\] taking each point on the sphere to its z coordinate, or the point's \textit{height} assuming the z-axis is aligned in the vertical direction.

The flow lines of this dynamical system are just the meridians, connecting the poles, the critical points of the function $z$ on the $2$-sphere. 
We take $N=400$ segments of  trajectories on the sphere, that start at randomly drawn points. These trajectories are plotted in Figure \ref{fig:sphere}.

\begin{wrapfigure}[14]{r}{0.63\textwidth}
\centering
\includegraphics[height=1.4in]{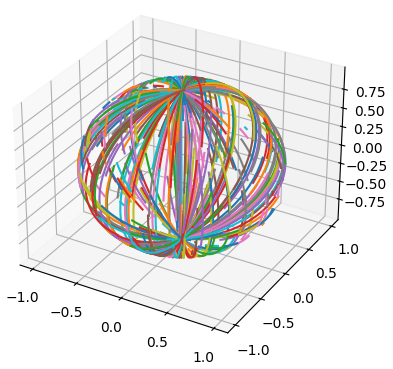}
\includegraphics[height=1.4in]{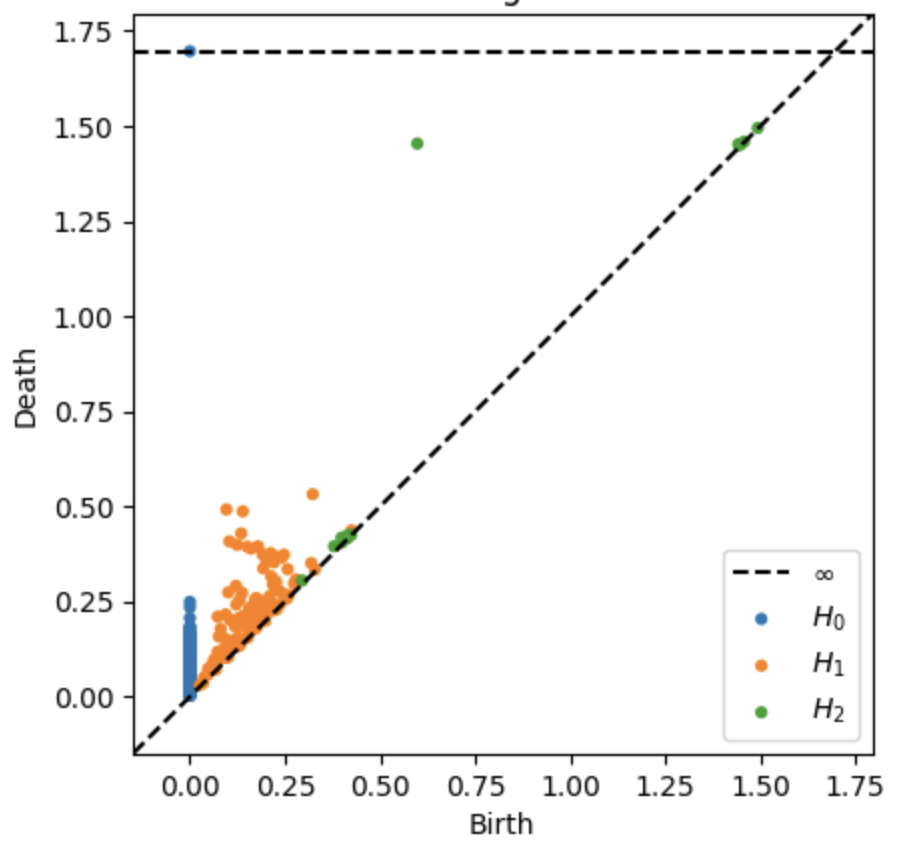}
\caption{\small Segments of trajectories of the gradient of the height function on the sphere. Right: Persistence diagram generated from these trajectories on the sphere. Here $N=400, n=15, t=10, T=1.5$.}\label{fig:sphere}
\end{wrapfigure}

 The persistence diagram generated by the MSSD algorithm for this set of trajectories is also shown in Figure \ref{fig:sphere}. The persistence diagram clearly reflects the fact that the homologies of a sphere $\Sph^d$ have rank $1$ in dimensions $0$ and $d$, and vanish elsewhere.

\subsection{Gradient Dynamical System: Random Trigonometric Function on the Torus}
We consider the example of a gradient dynamical system on the 2-torus $\mathbb{T}^2 \subset \R^3$ embedded in $\R^3$. We can parametrize the torus with two parameters $(\theta, \phi) \in [0, 2\pi] \times [0, 2\pi]$.


\begin{wrapfigure}[14]{r}{0.63\textwidth}
\centering
\includegraphics[height=1.4in, width=.3\textwidth]{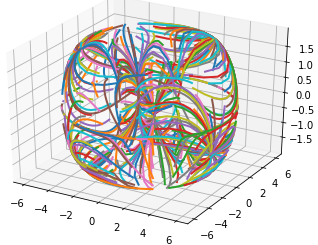}
\includegraphics[height=1.3in, width=.3\textwidth]{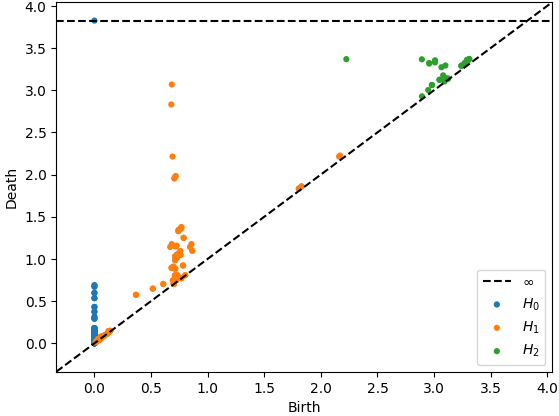}
\caption{\small Segments of trajectories of random vector field on 2-torus. Right: Persistence diagram generated from these segments. Here $N=400, n=25, t=3, T=2.5$}
\label{fig:torus}
\end{wrapfigure}

We used a random finite Fourier polynomial 
\(
    f(\theta, \phi) = \sum_{i=0}^{k}\sum_{j=0}^{k} a_{i,j}\cos((i+1)\theta+\theta_{i,j})cos((j+1)\phi+\phi_{i,j})
\)
where the coefficients $a_{i,j}$ are sampled from the standard normal distribution. We use $k = 2$ in this article.

%

We generated $N= 400$ trajectories, sampled from random initial conditions. Along each trajectory, $n=25$ samples were taken, for $T=2.5$. In this experiment, we took $V=\Real^3$, with the observables given by an embedding of the torus into the $3$-dimensional Euclidean space. 

The resulting trajectories are plotted in the left display of the Figure \ref{fig:torus}. The persistence diagram corresponding to $t=3$ shown on the right display of the Figure \ref{fig:torus}. 

The homology groups of the $2$-torus are of rank $1$ in the dimensions $0$ and $2$, and $2$ in dimension $1$. This is clearly captured in the persistence diagram Figure \ref{fig:torus}.

%

To check our algorithm for scalar output, we considered a random observation function given by 
\(g(\theta, \phi) = b_0 x(\theta, \phi) + b_1 y(\theta, \phi) + b_2 z(\theta, \phi)
\)
where the coefficients $(b_0, b_1, b_2)$ are sampled from a Gaussian distribution and the functions $x,y,z$ are Cartesian coordinates for our embedding of the $2$-torus in $\Real^3$.

\begin{wrapfigure}[13]{r}{0.63\textwidth}
\centering
\includegraphics[height=1.4in, width=.3\textwidth]{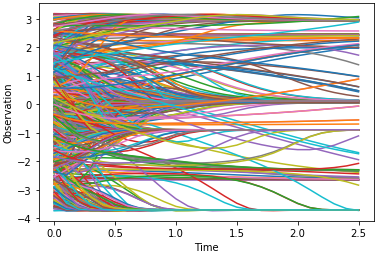}
\includegraphics[height=1.4in, width=.3\textwidth]{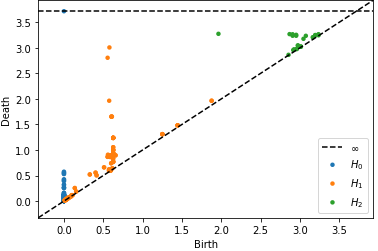}
\caption{\small Segments of observations of trajectories of random vector field on 2-torus. Right: Persistence diagram generated from these segments. Here $N=650, n=25, t=1, T=2.5$}
\label{fig:torus_emb}
\end{wrapfigure}

As a result, we get a set of time series plotted on the left display of the Figure \ref{fig:torus_emb}. 
The persistence diagram is shown on the right, still clearly matching the homology of $\mathbb{T}^2$.

\subsection{Chaotic Dynamical System: Lorenz System}

We now look at the standard example of a chaotic dynamical system given by the Lorenz dynamical system in $\R^3$ \citep{gilmore}:\begin{align*}
\dot{x} &= \sigma (y - x) \\
\dot{y} &= x(\rho - z) - y \\
\dot{z} &= xy - \beta z
\end{align*}
Here $\sigma, \rho, \beta$ are parameters which we set to $10, 28, \frac{8}{3}$ respectively. 

\begin{wrapfigure}[13]{r}{0.6\textwidth}
\centering
\includegraphics[height=1.4in, width=.28\textwidth]{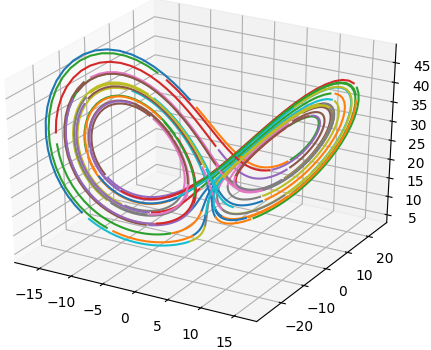}
\includegraphics[height=1.4in, width=.28\textwidth]{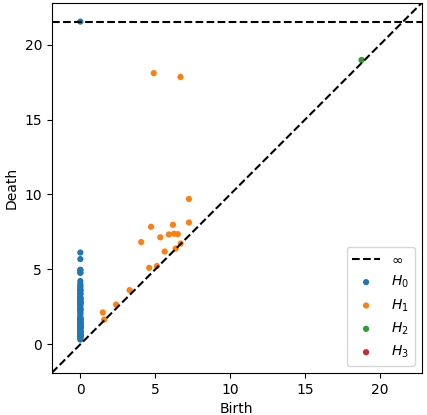}
\caption{\small Segments of trajectories of the Lorenz system. Right: Persistence diagram generated from these segments. Here $N=100, n=25, t=20, T=0.25$}
\label{fig:lorenz}
\end{wrapfigure}

It is known that the system exhibits chaotic behaviour for these values of parameters. We simulate one long trajectory from this dynamical system and split it into 100 pieces (the hyperbolicity of the system makes this essentially equivalent to randomly sampling shorter trajectories). The resulting trajectories are plotted on the left display of the Figure \ref{fig:lorenz}. 

For the slack distance derived from the tautological output function $g=(x,y,z)$, the persistence diagram is shown as the right display of the Figure \ref{fig:lorenz}. 

These results require some interpretation. While the ambient space carrying the vector field is $\Real^3$, the trajectory is essentially supported by a $2$-dimensional {\em template} \citep{birman1983}, homotopy equivalent to the wedge of two circles. As we do not sample trajectories starting far away from this template, it is not surprising we recover its topology: indeed, the ranks of the homology groups should be $1$ in the dimension $0$, $2$ in the dimension $1$, and vanish otherwise, in agreement with the right display on  the Figure \ref{fig:lorenz}.

We also repeat the experiments after passing the trajectories through a random polynomial output of degree $3$
with coefficients sampled from a Gaussian distribution. This gives us a set of time series observations as shown in Figure \ref{fig:lorenz_emb}. The corresponding persistence diagram still matches the homology of a wedge of two circles.

\begin{wrapfigure}[13]{r}{0.62\textwidth}
\centering 
\includegraphics[height=1.4in, width=.28\textwidth]{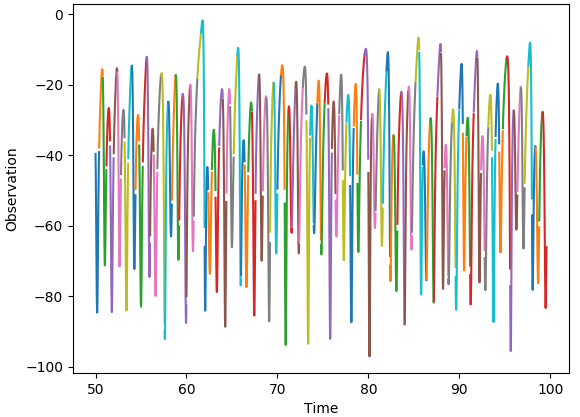}
\includegraphics[height=1.4in, width=.28\textwidth]{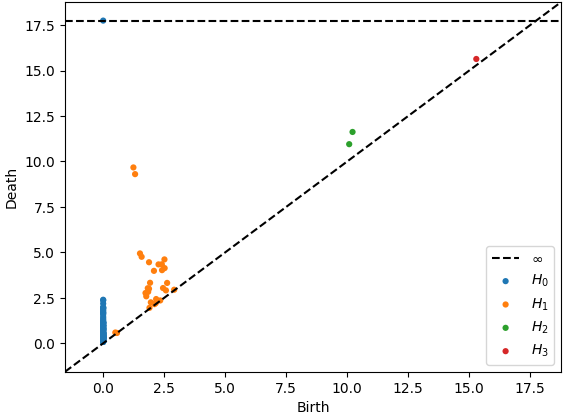}
\caption{\small Segments of observations of trajectories of the Lorenz system. Right: Persistence diagram generated from these segments. Here $N=150, n=25, t=10, T=0.25$}
\label{fig:lorenz_emb}
\end{wrapfigure}

The results become even more interesting when we increase the length $T$ of the fragments of trajectories. When this time $T$ is large enough, the segments of trajectories will hit the template's barrier (where the flow bifurcates to two different branches in the template) twice. Conceptually, this would lead to a restructured template that will now have the homotopy type of the wedge of $4$ circles. 

\begin{wrapfigure}[14]{r}{0.6\textwidth}
\centering
\includegraphics[height=1.4in, width=.28\textwidth]{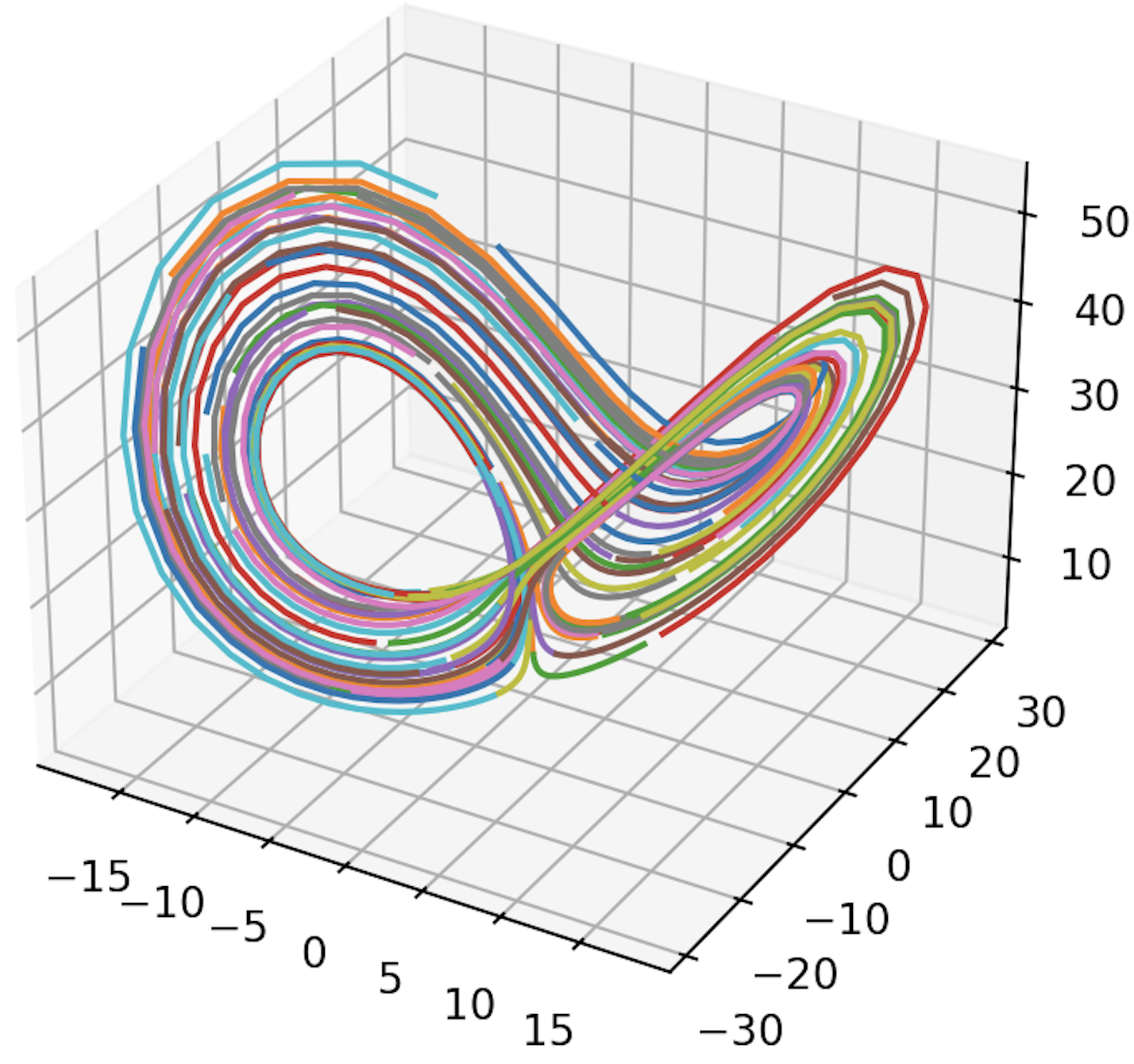}
\includegraphics[height=1.4in, width=.28\textwidth]{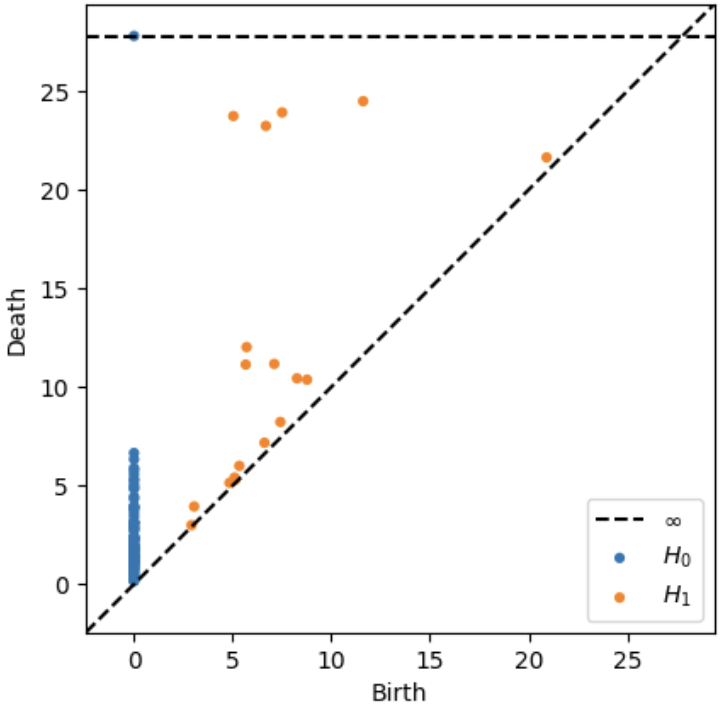}
\caption{\small Segments of longer trajectories of the Lorenz system. Right: Persistence diagram generated from these segments. Here $N=100, n=25, t=20, T=0.5$}
\label{fig:lorenz_long}
\end{wrapfigure}

And indeed, this is exactly what we observe in the Figure \ref{fig:lorenz_long}. 

Perhaps the most intuitive way to think about this phenomenon is to compare with the picture a topological data analysis would paint of a Cantor set $X\subset [0,1]$ obtained by iteratively removing the middle third from the intervals, starting with the unit one: at the low resolution $1/18\leq r<1/6$, the union of $r$-balls with centers at the points of $X$ are homotopy equivalent to two points; at the resolution $1/54\leq r<1/18$, that is $4$ points and so on. Given that the attractor of the the dynamical systems of Lorentz type are cantori, this doubling of the corresponding homologies is not unexpected after all.

\subsection*{Acknowledgements}
The authors were partially supported by the AFOSR MURI {\em Hybrid Dynamics - Deconstruction and Aggregation}.

\newpage
\bibliography{ma_ymb}

\begin{thebibliography}{17}
\providecommand{\natexlab}[1]{#1}
\providecommand{\url}[1]{\texttt{#1}}
\expandafter\ifx\csname urlstyle\endcsname\relax
  \providecommand{\doi}[1]{doi: #1}\else
  \providecommand{\doi}{doi: \begingroup \urlstyle{rm}\Url}\fi

\bibitem[Birman and Williams(1983)]{birman1983}
Joan~S Birman and Robert~F Williams.
\newblock Knotted periodic orbits in dynamical systems i: Lorenz’s equations.
\newblock \emph{Topology}, 22\penalty0 (1):\penalty0 47--82, 1983.

\bibitem[Brunton et~al.(2016)Brunton, Proctor, and
  Kutz]{brunton2016discovering}
Steven~L. Brunton, Joshua~L. Proctor, and J.~Nathan Kutz.
\newblock Discovering governing equations from data by sparse identification of
  nonlinear dynamical systems.
\newblock \emph{Proceedings of the National Academy of Sciences}, 113:\penalty0
  3932--3937, 2016.

\bibitem[Budišić et~al.(2012)Budišić, Mohr, and Mezić]{budisic12}
Marko Budišić, Ryan Mohr, and Igor Mezić.
\newblock Applied {K}oopmanism.
\newblock \emph{Chaos}, 22\penalty0 (4):\penalty0 047510, December 2012.
\newblock ISSN 1054-1500.
\newblock \doi{10.1063/1.4772195}.
\newblock URL \url{https://doi.org/10.1063/1.4772195}.

\bibitem[Champion et~al.(2019)Champion, Lusch, Kutz, and
  Brunton]{champion2019data}
Kathleen Champion, Bethany Lusch, J.~Nathan Kutz, and Steven~L. Brunton.
\newblock Data-driven discovery of coordinates and governing equations.
\newblock \emph{Proceedings of the National Academy of Sciences}, 116:\penalty0
  22445--22451, 2019.

\bibitem[Cohen et~al.(1995)Cohen, Jones, and Segal]{Coh95}
Ralph~L Cohen, John~DS Jones, and Graeme~B Segal.
\newblock Morse theory and classifying spaces.
\newblock \emph{preprint}, 1995.

\bibitem[Edelsbrunner and Harer(2010)]{harer2010}
Herbert Edelsbrunner and John Harer.
\newblock \emph{Computational topology: an introduction}.
\newblock American Mathematical Soc., 2010.

\bibitem[Ferry and Okun(1995)]{FerryO}
Steven~C Ferry and Boris~L Okun.
\newblock Approximating topological metrics by riemannian metrics.
\newblock \emph{Proceedings of the American Mathematical Society}, 123\penalty0
  (6):\penalty0 1865--1872, 1995.

\bibitem[Gilmore(1998)]{gilmore}
Robert Gilmore.
\newblock Topological analysis of chaotic dynamical systems.
\newblock \emph{Reviews of Modern Physics}, 70\penalty0 (4):\penalty0
  1455--1529, October 1998.
\newblock ISSN 0034-6861, 1539-0756.
\newblock \doi{10.1103/RevModPhys.70.1455}.
\newblock URL \url{https://link.aps.org/doi/10.1103/RevModPhys.70.1455}.

\bibitem[Jakubczyk(1980)]{Jac80}
Bronis\l{}aw Jakubczyk.
\newblock Existence and uniqueness of realizations of nonlinear systems.
\newblock \emph{SIAM Journal on Control and Optimization}, 18\penalty0
  (4):\penalty0 455--471, 1980.
\newblock \doi{10.1137/0318034}.

\bibitem[Latschev(2001)]{Lat}
J.~Latschev.
\newblock Vietoris-{Rips} complexes of metric spaces near a closed {Riemannian}
  manifold:.
\newblock \emph{Archiv der Mathematik}, 77\penalty0 (6):\penalty0 522--528,
  December 2001.
\newblock ISSN 0003-889X.
\newblock \doi{10.1007/PL00000526}.
\newblock URL \url{http://link.springer.com/10.1007/PL00000526}.

\bibitem[Otto et~al.(2023)Otto, Macchio, and Rowley]{otto2023learning}
Samuel~E. Otto, Gregory~R. Macchio, and Clarence~W. Rowley.
\newblock Learning nonlinear projections for reduced-order modeling of
  dynamical systems using constrained autoencoders.
\newblock \emph{Chaos: An Interdisciplinary Journal of Nonlinear Science}, 33,
  2023.

\bibitem[Perea and Harer(2015)]{Per15}
Jose~A. Perea and John Harer.
\newblock Sliding windows and persistence: An application of topological
  methods to signal analysis.
\newblock \emph{Foundations of Computational Mathematics}, 15\penalty0
  (3):\penalty0 799--838, 2015.

\bibitem[Riehl(2017)]{Riehl}
Emily Riehl.
\newblock \emph{Category theory in context}.
\newblock Courier Dover Publications, 2017.

\bibitem[Schmidt and Lipson(2009)]{schmidt2009distilling}
Michael Schmidt and Hod Lipson.
\newblock Distilling free-form natural laws from experimental data.
\newblock \emph{Science}, 324:\penalty0 81--85, 2009.

\bibitem[Schutter(2000)]{schutter_minimal_2000}
B.~De Schutter.
\newblock Minimal state-space realization in linear system theory: an overview.
\newblock \emph{Journal of Computational and Applied Mathematics}, 121\penalty0
  (1):\penalty0 331--354, September 2000.
\newblock ISSN 0377-0427.
\newblock \doi{10.1016/S0377-0427(00)00341-1}.
\newblock URL
  \url{https://www.sciencedirect.com/science/article/pii/S0377042700003411}.

\bibitem[Takens(2006)]{Tak06}
Floris Takens.
\newblock Detecting strange attractors in turbulence.
\newblock In \emph{Dynamical Systems and Turbulence, Warwick 1980: proceedings
  of a symposium held at the University of Warwick 1979/80}, pages 366--381.
  Springer, 2006.

\bibitem[van~der Schaft(1986)]{schaft86}
A.~J. van~der Schaft.
\newblock On realization of nonlinear systems described by higher-order
  differential equations.
\newblock \emph{Mathematical systems theory}, 19\penalty0 (1):\penalty0
  239--275, December 1986.
\newblock ISSN 1433-0490.
\newblock \doi{10.1007/BF01704916}.
\newblock URL \url{https://doi.org/10.1007/BF01704916}.

\end{thebibliography}

\end{document}